\documentclass{article}
\usepackage{amsmath}
\usepackage{amsthm}
\usepackage{amsfonts}
\newtheorem{theorem}{Theorem}

\begin{document}
\title{Mehler-Heine asymptotics for multiple orthogonal polynomials 
\thanks{This research was supported by KU Leuven Research Grant OT/12/073, FWO Research Grant G.0934.13
and the Belgian Interuniversity Attraction Poles Programme P7/18.}}
\author{Walter Van Assche}
\maketitle
\begin{abstract}
Mehler-Heine asymptotics describe the behavior of orthogonal polynomials near the edges of the interval where
the orthogonality measure is supported. For Jacobi polynomials and Laguerre polynomials this asymptotic behavior near the hard edge involves Bessel functions $J_\alpha$. We show that the asymptotic behavior near the endpoint of
the interval of (one of) the measures for multiple orthogonal polynomials involves a generalization of the Bessel
function. The multiple orthogonal polynomials considered are Jacobi-Angelesco polynomials, Jacobi-Pi\~neiro polynomials, multiple Laguerre polynomials,
multiple orthogonal polynomials associated with modified Bessel functions (of the first and second kind), and multiple orthogonal polynomials associated with Meijer $G$-functions.  
\end{abstract}

\section{Introduction}
Mehler (1868) and Heine (1878) obtained the asymptotic behavior of Legendre polynomials near the endpoint of their interval of orthogonality:
\[  \lim_{n \to \infty} P_n\left( \cos \frac{z}{n} \right) 
  = \lim_{n \to \infty} P_n\left( 1 - \frac{z^2}{2n^2} \right)
  = J_0(z), \]
where $J_0$ is the Bessel function of order $0$. This result can easily be extended to Jacobi polynomials
\cite[Thm.~8.1.1]{Szego}
\[  \lim_{n \to \infty} n^{-\alpha} P_n^{(\alpha,\beta)} \left( \cos \frac{z}{n} \right) 
  = \lim_{n \to \infty} n^{-\alpha} P_n^{(\alpha,\beta)}\left( 1 - \frac{z^2}{2n^2} \right) = (z/2)^{-\alpha} J_\alpha(z), \]
where $J_\alpha$ is the Bessel function of order $\alpha$. Both asymptotic formulas hold uniformly on every compact subset of the complex plane.  
This asymptotic behavior near the so-called hard edge gives as an important consequence detailed information of the zeros of Jacobi polynomials near the endpoint $1$: if $x_{k,n}=\cos \theta_{k,n}$ $(1 \leq k \leq n)$ are the zeros of the Jacobi polynomial $P_n^{(\alpha,\beta)}$ in decreasing order, then
for fixed $k$ one has
\[   \lim_{n \to \infty} n \theta_{k,n} = j_k,  \]
where $j_k$ is the $k$th positive zero of $J_\alpha$. The asymptotic behavior near the other endpoint $-1$ can easily be obtained using the relation
$P_n^{(\alpha,\beta)}(-x) = (-1)^n P_n^{(\beta,\alpha)}(x)$.
A similar asymptotic formula holds for Laguerre polynomials near the endpoint $0$ \cite[Thm.~8.1.3]{Szego}
\[  \lim_{n \to \infty} n^{-\alpha} L_n^{(\alpha)}(z/n) = z^{-\alpha/2} J_{\alpha}(2\sqrt{z}), \]
uniformly on compact subsets of $\mathbb{C}$.    

In this paper we will obtain asymptotic formulas of Mehler-Heine type for some classical multiple orthogonal polynomials \cite[Ch.\ 23]{Ismail} 
\cite{AptBranWVA}. 
These are polynomials $P_{\vec{n}}$ with a multi-index $\vec{n}=(n_1,\ldots,n_r) \in \mathbb{N}^r$ of degree $|\vec{n}|=n_1+\cdots+n_r$, 
which satisfy orthogonality conditions with respect to $r$ classical weights on the real
line:
\[   \int P_{\vec{n}}(x) x^k w_j(x)\, dx = 0, \qquad 0 \leq k \leq n_j-1,\ 1 \leq j \leq r. \]
Such polynomials are used in the analysis of random matrices with external field, products of random matrices \cite{KuijlZhang}, non-intersecting Brownian motions, and various other determinantal point processes \cite{Kuijl2}.

\section{Multiple orthogonal polynomials}

\subsection{Jacobi-Angelesco polynomials}
The Jacobi-Angelesco polynomials are given by  \cite[\S 23.3.1]{Ismail}
\begin{equation}  \label{RodJacAng}
 (1+x)^\alpha x^\beta (1-x)^\gamma P_{n,n}(x) = (-1)^n \frac{d^n}{dx^n} \left( (1+x)^{n+\alpha} x^{n+\beta} (1-x)^{n+\gamma} \right). 
\end{equation}
Note that we used a different normalization and that these polynomials are not monic but have leading coefficient $(\alpha+\beta+\gamma+1)_{3n}/(\alpha+\beta+\gamma+1)_{2n}$.
For $\alpha,\beta,\gamma >-1$ they satisfy orthogonality relations on two touching intervals $[-1,0]$ and $[0,1]$:
\begin{align*}  
 \int_{-1}^0  P_{n,n}(x) x^k (1+x)^\alpha |x|^\beta (1-x)^\gamma \, dx &= 0, \qquad 0 \leq k \leq n-1, \\
 \int_0^1 P_{n,n}(x) x^k (1+x)^\alpha |x|^\beta (1-x)^\gamma \, dx &= 0, \qquad 0 \leq k \leq n-1. 
\end{align*}
The following theorem was proved in \cite{Takata2} and \cite[Thm. 1.4]{Tulyakov}. Their proofs are based on a differential equation or a recurrence relations for these polynomials. We give an alternative proof.

\begin{theorem}  \label{thm:JacAng}
Suppose $P_{n,n}$ are the Jacobi-Angelesco polynomials, then
\[   \lim_{n \to \infty} (-1)^n P_{n,n}(z/n^{3/2}) = {}_0F_2(-;\frac{\beta+1}{2},\frac{\beta}{2}+1;-z^2/4) 
		= \sum_{k=0}^\infty (-1)^k \frac{z^{2k}}{(\beta+1)_{2k} k!} ,  \]
uniformly on compact subsets of $\mathbb{C}$. 
\end{theorem}

\begin{proof}
If we expand $(1+x)^{n+\alpha}$ and $(1-x)^{n+\gamma}$ as a power series, then the Rodrigues formula
\eqref{RodJacAng} becomes
\[   (1+x)^\alpha x^\beta (1-x)^\gamma P_{n,n}(x) 
  = (-1)^n \sum_{k=0}^\infty \sum_{j=0}^\infty \binom{n+\alpha}{k} \binom{n+\gamma}{j} (-1)^j \frac{d^n}{dx^n} x^{n+k+j+\beta}, \]
which readily gives
\[   P_{n,n}(x) = (1+x)^{-\alpha}(1-x)^{-\gamma} (-1)^n \sum_{k=0}^\infty \sum_{j=0}^\infty \binom{n+\alpha}{k} \binom{n+\gamma}{j} (-1)^j
            \frac{(n+\beta+1)_{k+j}}{(\beta+1)_{k+j}} x^{k+j}.  \]
This series expansion holds for $|x| < 1$ (when $\alpha$ and $\gamma$ are integers, then the series terminate). Let $k+j=\ell$, then the expansion is (after changing the order of summation)
\[   P_{n,n}(x) = (1+x)^{-\alpha}(1-x)^{-\gamma} (-1)^n  
 \sum_{\ell=0}^\infty d_\ell(n)     (-1)^\ell \frac{(n+\beta+1)_\ell}{(\beta+1)_\ell} x^\ell,  \]
where 
\[   d_\ell(n) = \sum_{k=0}^\ell (-1)^k \binom{n+\alpha}{k} \binom{n+\gamma}{\ell-k}. \]
This sum is nearly the Chu-Vandermonde sum except that it contains the alternating term $(-1)^k$. This $d_{\ell}(n)$ is the convolution
$(a\ast b)_\ell$ of the sequences $a_k = (-1)^k \binom{n+\alpha}{k}$ and  $b_k = \binom{n+\gamma}{k}$, and as such $d_\ell(n)$ is the $\ell$-th
term in the series expansion of $(1-z^2)^n(1-z)^\alpha(1+z)^\gamma$. If we write
\[    (1-z)^\alpha(1+z)^\gamma = \sum_{k=0}^\infty c_k z^k, \qquad |z| < 1, \]
then
\[  d_\ell(n) = \sum_{k=0}^{\lfloor \ell/2 \rfloor} (-1)^k \binom{n}{k} c_{\ell-2k}.  \]
This formula shows that $d_\ell(n) = \mathcal{O}(n^{\lfloor \ell/2 \rfloor})$, and more precisely
\begin{equation} \label{dinfty}
   \lim_{n \to \infty} \frac{1}{n^{\ell/2}} d_\ell(n) = \begin{cases}
                            0, & \textrm{if $\ell$ is odd}, \\
                           (-1)^{\ell/2}/(\ell/2)!, & \textrm{if $\ell$ is even}.
                             \end{cases} 
\end{equation}
Now we take $x=z/n^{3/2}$, then if $z$ is in a compact subset of $\mathbb{C}$ one has $|z|/n^{3/2} < r < 1$ for sufficiently large $n$, so that
the series for $P_{n,n}(z/n^{3/2})$ can be used. Now $(1+z/n^{3/2})^\alpha \to 1$ and $(1-z/n^{3/2})^\gamma \to 1$ uniformly for $z$ on a compact
set in the complex plane, and the convergence of the terms in the series
\[   \lim_{n \to \infty} \frac{1}{n^\ell} (n+\beta+1)_\ell = 1 \]
and \eqref{dinfty} gives
\[   \lim_{n \to \infty} (-1)^n P_{n,n}(z/n^{3/2}) = \sum_{\ell=0}^\infty (-1)^{\ell} \frac{x^{2\ell}}{\ell! (\beta+1)_{2\ell}}.  \]
The interchanging of the limit with the infinite sum can be justified by using Lebesgue's dominated convergence theorem.
The series in the limit can be expressed as a hypergeometric function by using 
\[   (\beta+1)_{2\ell} = 2^{2\ell}(\frac{\beta+1}{2})_\ell (\frac{\beta}{2}+1)_\ell.  \] 
\end{proof}

\subsection{Jacobi-Pi\~neiro polynomials}
These are defined by the Rodrigues formula \cite[\S 23.3.2]{Ismail}
\[   (-1)^{|\vec{n}|} \prod_{j=1}^r (|\vec{n}|+\alpha_j+\beta+1)_{n_j} (1-x)^\beta P_{\vec{n}}(x) 
=  \prod_{j=1}^r \left( x^{-\alpha_j} \frac{d^{n_j}}{dx^{n_j}} x^{n_j+\alpha_j} \right) (1-x)^{|\vec{n}|+\beta}. \]
For $\alpha_1,\ldots,\alpha_r,\beta >-1$ and $\alpha_i - \alpha_j \notin \mathbb{Z}$ $(i \neq j)$ they satisfy orthogonality relations on one interval $[0,1]$
with respect to $r$ Jacobi weights
\[  \int_0^1 P_{\vec{n}}(x) x^k x^{\alpha_j}(1-x)^\beta\, dx = 0, \qquad 0 \leq k \leq n_j-1, \ 1 \leq j \leq r.  \]
The Mehler-Heine formula near the endpoint $0$ is:

\begin{theorem}  \label{thm:JacPin}
Suppose $n_j = \lfloor q_j n\rfloor$, where $q_j >0$ and $q_1+\cdots+q_r = 1$, so that $|\vec{n}|/n \to 1$ as $n \to \infty$.
Then for the Jacobi-Pi\~neiro polynomials $P_{\vec{n}}$ one has
\begin{multline*}
  \lim_{n \to \infty} (-1)^{|\vec{n}|} \prod_{j=1}^r \frac{(|\vec{n}|+\alpha_j+\beta+1)_{n_j}}{(\alpha_j+1)_{n_j}}
 P_{\vec{n}}(z/n^{r+1}) \\
  = {}_0F_r (-;\alpha_1+1,\ldots,\alpha_r+1;-q_1\cdots q_r z), 
\end{multline*}
uniformly on compact subsets of $\mathbb{C}$.
\end{theorem}

\begin{proof}
An explicit formula for the Jacobi-Pi\~neiro polynomials is \cite[Eq. (23.3.5)]{Ismail}
\begin{multline*}
  (-1)^{|\vec{n}|} \prod_{j=1}^r (|\vec{n}|+\alpha_j+\beta+1)_{n_j}  (1-x)^{\beta} P_{\vec{n}}(x) \\
  = \prod_{j=1}^r (\alpha_j+1)_{n_j}\  {}_{r+1}F_r\left( \left. \begin{array}{c} 
                  -|\vec{n}|-\beta,\alpha_1+n_1+1,\ldots,\alpha_r+n_r+1 \\
                     \alpha_1+1, \ldots,\alpha_r+1 \end{array} \right| x \right). 
\end{multline*}
Take $x=z/n^{r+1}$, then $(1-z/n^{r+1})^\beta \to 1$ uniformly on compact subsets of $\mathbb{C}$. Furthermore one has
\[   \lim_{n \to \infty} \frac{(\alpha_j+n_j+1)_k}{n^k}= q_j^k, \quad \lim_{n\to\infty} \frac{(-|\vec{n}|-\beta)_k}{n^k} = (-1)^k, \]
so that all the terms in the hypergeometric series converge. Lebesgue's dominated convergence theorem then gives the required result. 
\end{proof}

\subsection{Multiple Laguerre polynomials}
There are two kinds of multiple Laguerre polynomials. Multiple Laguerre polynomials of the first kind are given by the
Rodrigues formula
\[   (-1)^{|\vec{n}|} e^{-x} L_{\vec{n}}^{[1]}(x) = \prod_{j=1}^r \left( x^{-\alpha_j} \frac{d^{n_j}}{dx^{n_j}} x^{n_j+\alpha_j} \right) e^{-x}. \]
If $\alpha_1,\ldots,\alpha_r > -1$ and $\alpha_i - \alpha_j \notin \mathbb{Z}$ $(i \neq j)$ , then they satisfy the orthogonality relations
\[   \int_0^\infty L_{\vec{n}}^{[1]}(x) x^k x^{\alpha_j} e^{-x}\, dx = 0, \qquad 0 \leq k \leq n_j-1, \ 1 \leq j \leq r. \]
The weights $w_j$ all have different behavior $\mathcal{O}(x^{\alpha_j})$ as $x \to 0$. This is reflected in the following
Mehler-Heine formula.

\begin{theorem}  \label{thm:Ln1} 
Suppose $n_j = \lfloor q_j n\rfloor$, where $q_j >0$ and $q_1+\cdots +q_r = 1$, so that $|\vec{n}|/n \to 1$ as $n \to \infty$.
For the multiple Laguerre polynomials of the first kind one then has
\[  \lim_{n \to \infty} (-1)^{|\vec{n}|} \frac{L_{\vec{n}}^{[1]}(z/n^r)}{\prod_{j=1}^r (\alpha_j+1)_{n_j}}
   = {}_0F_r (-;\alpha_1+1,\ldots,\alpha_r+1; -(q_1\cdots q_r)z), \]
uniformly on compact subsets of $\mathbb{C}$.
\end{theorem}
\begin{proof}
These multiple Laguerre polynomials of the first kind are given in terms of generalized hypergeometric functions as
\[   (-1)^{|\vec{n}|} e^{-x} L_{\vec{n}}^{[1]}(x) \\ 
   = \prod_{j=1}^r (\alpha_j+1)_{n_j} \ 
  {}_rF_r \left( \begin{array}{c} \alpha_1+n_1+1,\ldots,\alpha_r+n_r+1 \\
         \alpha_1+1,\ldots,\alpha_r+1  \end{array} ;-x \right),  \]
see \cite[\S 23.4.1]{Ismail}. Now take $\vec{n}=(\lfloor q_1n\rfloor,\ldots,\lfloor q_rn\rfloor)$  and let $x=z/n^r$. Then when $n \to \infty$, 
each term in the hypergeometric series converges, since
\[  \lim_{n \to \infty} \frac{(\alpha_j+n_j+1)_k}{n^k} = q_j^k. \]
Lebesgue's dominated convergence theorem and $\lim_{n \to \infty} e^{-z/n^r} =1$, then readily give the required asymptotic formula.
\end{proof}

Multiple Laguerre polynomials of the second kind are given by
\[  (-1)^{|\vec{n}|} \prod_{j=1}^r c_j^{n_j} \ x^\alpha L_{\vec{n}}^{[2]}(x) =
   \prod_{j=1}^r \left( e^{c_j x} \frac{d^{n_j}}{dx^{n_j}} e^{-c_jx} \right) x^{|\vec{n}|+\alpha}.  \]
When $\alpha > -1$ and $c_1,\ldots,c_j >0$ and $c_i \neq c_j$ whenever $i \neq j$, these polynomials satisfy the following orthogonality relations
\[   \int_0^\infty L_{\vec{n}}^{[2]}(x) x^k x^{\alpha} e^{-c_jx}\, dx = 0, \qquad 0 \leq k \leq n_j-1,\ 1 \leq j \leq r.  \]
Note that these $r$ weights have the same behavior $\mathcal{O}(x^\alpha)$ as $x \to 0$ but their asymptotic behavior near $\infty$ is different.
The Mehler-Heine asymptotic formula is now in terms of the Bessel function. 

\begin{theorem}  \label{thm:Ln2}
Suppose $n_j = \lfloor q_j n\rfloor$, where $q_j >0$ and $q_1+\cdots+q_r = 1$, so that $|\vec{n}|/n \to 1$ as $n \to \infty$.
For the multiple Laguerre polynomials of the second kind one then has
\begin{multline*}
   \lim_{n \to \infty} (-1)^{|\vec{n}|} c_1^{n_1}\cdots c_r^{n_r} \frac{L_{\vec{n}}^{[2]}(z/n)}{(\alpha + 1)_{|\vec{n}|}}  \\ 
  =  \bigl((q_1c_1+\cdots+q_rc_r)z\bigr)^{-\alpha/2}J_{\alpha}(2\sqrt{(q_1c_1+\cdots+q_rc_r)z}), 
\end{multline*} 
uniformly on compact subsets of $\mathbb{C}$.
\end{theorem}
\begin{proof}
An explicit expression of the multiple Laguerre polynomials of the second kind is \cite[Eq. (23.4.5)]{Ismail}
\[  L_{\vec{n}}^{[2]}(x) = \sum_{k_1=0}^{n_1} \cdots \sum_{k_r=0}^{n_r} \binom{n_1}{k_1} \cdots \binom{n_r}{k_r} \binom{|\vec{n}|+\alpha}{|\vec{k}|}
   (-1)^{|\vec{k}|} \frac{|\vec{k}|!}{c_1^{k_1}\cdots c_r^{k_r}} x^{|\vec{n}|-|\vec{k}|}.   \]
Change $k_j \to n_j-k_j$ to find
\[  (-1)^{|\vec{n}|} L_{\vec{n}}^{[2]}(x) = \sum_{k_1=0}^{n_1} \cdots \sum_{k_r=0}^{n_r} \binom{n_1}{k_1} \cdots \binom{n_r}{k_r} 
  \binom{|\vec{n}|+\alpha}{|\vec{n}|-|\vec{k}|}
   (-1)^{|\vec{k}|} \frac{(|\vec{n}|-|\vec{k}|)!}{c_1^{n_1-k_1}\cdots c_r^{n_r-k_r}} x^{|\vec{k}|}.   \]
Now take $\vec{n}=(\lfloor q_1n\rfloor,\ldots,\lfloor q_rn\rfloor)$  and  $x=z/n$, then as $n \to \infty$ one can use
\[   \lim_{n \to \infty} n^{-k_j} \binom{n_j}{k_j} = q_j^{k_j}, \quad  
  (|\vec{n}|-|\vec{k}|)! \binom{|\vec{n}|+\alpha}{|\vec{n}|-|\vec{k}|} = \frac{(\alpha+1)_{|\vec{n}|}}{(\alpha+1)_{|\vec{k}|}}, \]
to find
\[   \lim_{n \to \infty} (-1)^{|\vec{n}|} c_1^{n_1}\cdots c_r^{n_r} \frac{L_{\vec{n}}^{[2]}(z/n)}{(\alpha+1)_{|\vec{n}|}} 
   = \sum_{k_1=0}^\infty \cdots \sum_{k_r=0}^\infty (-1)^{|\vec{k}|} \frac{(q_1c_1)^{k_1}\cdots (q_rc_r)^{k_r}}{(\alpha+1)_{|\vec{k}|}} 
\frac{z^{|\vec{k}|}}{k_1!\cdots k_r!} , \]
uniformly on compact subsets of $\mathbb{C}$. The limit can also be written as
\[   \sum_{k=0}^\infty (-1)^k \frac{z^k}{(\alpha+1)_k k!}  \sum_{|\vec{k}|=k} \frac{k!}{k_1!\cdots k_r!} (q_1c_1)^{k_1}\cdots (q_rc_r)^{k_r}, \]
but by the multinomial theorem
\[   \sum_{|\vec{k}|=k} \frac{k!}{k_1!\cdots k_r!} (q_1c_1)^{k_1}\cdots (q_rc_r)^{k_r} = (q_1c_1 + \cdots + q_rc_r)^k, \]
from which the result follows.
\end{proof}

Mehler-Heine asymptotics for other multiple Laguerre polynomials 
\begin{equation}  \label{So}
    x^p e^{-x^r} L_n(x,p) = \frac{1}{n!} \frac{d^n}{dx^n} \left( x^{n+p}e^{-x^r} \right), 
\end{equation}
were given in \cite{Sorokin} and \cite{Tulyakov}. These are polynomials of degree $rn$ with orthogonality conditions (for $p >-1$)
\[   \int_0^{\infty \omega^j} L_n(x,p) x^k x^p e^{-x^r}\, dx = 0, \qquad 0 \leq k \leq n-1, \ 0 \leq j \leq r-1, \]
where $\omega = e^{2\pi i/r}$. Observe that the orthogonality relations are not on the real line but on $r$ rays in the complex plane.
We do not give a proof but refer to \cite{Sorokin,Tulyakov}.

\begin{theorem}  \label{thm:SoTu}
For the multiple Laguerre polynomials given by \eqref{So} one has
\[  \lim_{n \to \infty}  n^{-p} L_n(x,p) = \frac{1}{\Gamma(p+1)} {}_0F_r(-;\frac{p+1}{r},\ldots,\frac{p+r}{r};-(z/r)^r),  \]
 uniformly on compact sets of $\mathbb{C}$.
\end{theorem}

\subsection{Multiple orthogonal polynomials for modified Bessel functions}
Multiple orthogonal polynomials associated with the modified Bessel functions $K_\nu$ and $K_{\nu+1}$ were introduced in \cite{WVAYaku}, 
see also \cite{BenDouak}. They satisfy the orthogonality relations
\begin{align*}
    \int_0^\infty x^k P_{n,m}(x) x^{\alpha+\nu/2} K_{\nu}(2\sqrt{x})\, dx &= 0, \qquad 0 \leq k \leq n-1, \\
    \int_0^\infty x^k P_{n,m}(x) x^{\alpha+(\nu+1)/2} K_{\nu+1}(2\sqrt{x})\, dx &= 0, \qquad 0 \leq k \leq m-1, 
\end{align*}

\begin{theorem}  \label{thm:Knu}
Let $p_{2n}(x)=P_{n,n}(x)$ and $p_{2n+1}(x) = P_{n+1,n}(x)$, where $P_{n,m}$ are the multiple orthogonal polynomials
associated with the modified Bessel functions $K_{\nu}$ and $K_{\nu+1}$. Then 
\[  \lim_{n \to \infty} (-1)^n \frac{p_n(z/n)}{(\alpha+1)_n(\alpha+\nu+1)_n} = {}_0F_2(-;\alpha+1,\alpha+\nu+1;-z), \]
uniformly on compact subsets of $\mathbb{C}$.
\end{theorem}

\begin{proof}
An explicit expression for these multiple orthogonal polynomials is \cite[Eq. (23)]{ElsWVA1}
\[   p_n(x) = (-1)^n (\alpha+1)_n(\alpha+\nu+1)_n \ {}_1F_2(-n;\alpha+1,\alpha+\nu+1;x).  \]
Take $x=z/n$ and use the fact that $(-n)_k/n^k \to (-1)^k$ as $n \to \infty$ to find the required result. The limit can be taken termwise
because one has $|(-n)_k| \leq n^k$ so that the sum
\[  \sum_{k=0}^n \frac{(-n)_k}{(\alpha+1)_k (\alpha+\nu+1)_k} \frac{(z/n)^k}{k!}  \]
is dominated by
\[   \sum_{k=0}^n \frac{1}{(\alpha+1)_k (\alpha+\nu+1)_k} \frac{|z|^k}{k!} = {}_0F_2(-;\alpha+1,\alpha+\nu+1;|z|)  \]
and hence Lebesgue's dominated convergence theorem can be used.
\end{proof}

There are also multiple orthogonal polynomials associated with the modified Bessel functions $I_\nu$ and $I_{\nu+1}$, see \cite{ElsWVA}
and \cite{Douak}. They satisfy (for $c > 0$)
\begin{align*}
    \int_0^\infty x^k P_{n,m}(x)  x^{\nu/2} e^{-cx} I_{\nu}(2\sqrt{x})\, dx &= 0, \qquad 0 \leq k \leq n-1, \\
    \int_0^\infty x^k P_{n,m}(x) x^{(\nu+1)/2} e^{-cx} I_{\nu+1}(2\sqrt{x})\, dx &= 0, \qquad 0 \leq k \leq m-1. 
\end{align*}
The Mehler-Heine asymptotics is somewhat different and is in terms of the Bessel function $J_\nu$. The following result was proved in 
\cite[Thm. 2]{ElsWVA2}.

\begin{theorem}  \label{thm:Inu}
Let $p_{2n}(x)=P_{n,n}(x)$ and $p_{2n+1}(x) = P_{n+1,n}(x)$, where $P_{n,m}$ are the multiple orthogonal polynomials
associated with the modified Bessel functions $I_{\nu}$ and $I_{\nu+1}$. Then 
\[   \lim_{n \to \infty} (-1)^n \frac{p_n(z/n)}{n^\nu n!} = e^{1/c} (cz)^{-\nu/2} J_\nu(2\sqrt{cz}), \]
uniformly on compact subsets of $\mathbb{C}$.
\end{theorem}

\subsection{Multiple orthogonal polynomials for Meijer $G$-functions}
Recently Kuij\-laars and Zhang \cite{KuijlZhang} introduced multiple orthogonal polynomials for weights which are Meijer $G$-functions
$G_{0,M}^{M,0}$, and they appear in investigating singular values of products of random matrices. They satisfy
\[    \int_0^\infty P_m(x) x^k w_j(x)\, dx = 0, \qquad 0 \leq k \leq \lceil \frac{m-j}{r} \rceil -1, \]
for $0 \leq j \leq r-1$, where the weights $w_j$ are Meijer $G$-functions
\begin{align*}
    w_j(x) &= G_{0,r}^{r,0} \left( \left. \begin{array}{c}  - \\ \nu_r, \nu_{r-1},\ldots,\nu_2,\nu_1+j \end{array} \right| x \right) \\ 
    &= \frac{1}{2\pi i} \int_{c-i\infty}^{c+i\infty} (s+\nu_1)_j \prod_{k=1}^r \Gamma(s+\nu_k) x^{-s}\, ds. 
\end{align*}
These are multiple orthogonal polynomials with multi-index $(n+1,\ldots,n+1,n,\ldots,n)$ when $m=nr+s$.
The special case $r=2$ corresponds to the multiple orthogonal polynomials for the modified Bessel functions $K_\nu$ and $K_{\nu+1}$
for $\nu=\nu_1-\nu_2$  and $\alpha=\nu_2$.

\begin{theorem}  \label{thm:G}
For the multiple orthogonal polynomials with the Meijer $G$-functions one has
\[  \lim_{n \to \infty}  (-1)^n \frac{P_n(z/n)}{\prod_{j=1}^r (\nu_j+1)_n} 
  =  {}_0F_r (-;\nu_1+1,\ldots,\nu_r+1; -z), \]
uniformly on compact subsets of $\mathbb{C}$.
\end{theorem}

\begin{proof}
The proof is a straightforward generalization of the proof of Theorem \ref{thm:Knu}.
These multiple orthogonal polynomials are hypergeometric polynomials \cite[Eq. (3.10)]{KuijlZhang}
\[  P_n(x) = (-1)^n \prod_{j=1}^r (\nu_j+1)_n \ {}_1F_r(-n;\nu_1+1,\ldots,\nu_r+1;x). \]
If we take $x=z/n$ and let $n \to \infty$, then $(-n)_k/n^k \to (-1)^k$ and the limit in this hypergeometric function
can be taken termwise. The result then follows immediately.
\end{proof}    

\section{Generalized Bessel functions}
The Mehler-Heine formulas for multiple orthogonal polynomials show that the function ${}_0F_r(-;\alpha_1+1,\ldots,\alpha_r+1;z)$
takes over the role of the Bessel function $J_\alpha$ for the local asymptotics near a hard edge. In the literature one uses the
terminology generalized Bessel function already for the entire function
\[  \phi(z) =  \sum_{k=0}^\infty \frac{z^k}{\Gamma(\rho k+\beta) k!}, \]
when $\rho > 0$ and $\beta$ any complex number (see, \cite{WongZhao,Wright}, and \cite[\S 2]{Sorokin}). Note that for $\rho=1$ one has
$J_{\beta-1}(t)=(t/2)^{\beta-1} \phi(-t^2/4)$ and for $\rho=r$ a positive integer, the multiplication formula for the gamma function,
\[    \Gamma(rz) = (\sqrt{2\pi})^{1-r} r^{rz-1/2} \prod_{j=0}^{r-1} \Gamma\left(z+\frac{j}{r}\right),  \]  
shows that $\phi(z)$ is, up to a multiplicative factor, equal to ${}_0F_r(-;\beta/r,(\beta+1)/r,\ldots, (\beta+r-1)/r;z/r^r)$.
So we may consider the entire function ${}_0F_r(-;\alpha_1+1,\ldots,\alpha_r+1;z)$ as a generalization of the Bessel function.
It satisfies the differential equation
\begin{equation}  \label{diffeq}
   \Theta (\Theta+\alpha_1)\cdots(\Theta+\alpha_r)y - zy = 0, \qquad \Theta = z \frac{d}{dz}, 
\end{equation}
which is a linear differential equation of order $r+1$ with a regular singular point at $z=0$ and an irregular singularity at infinity.
If no $\alpha_j$ is an integer and $\alpha_i-\alpha_j \notin \mathbb{Z}$ $(i \neq j)$, then a fundamental set of solutions of this differential equation is 
\[ y_0(z)={}_0F_r(-;\alpha_1+1,\ldots,\alpha_r+1;z) \]
together with
\[ y_j(z) = z^{-\alpha_j} {}_0F_r(-;1-\alpha_j,1+\alpha_1-\alpha_j,\ldots*\ldots,1+\alpha_r-\alpha_j;z), \qquad 1 \leq j \leq r, \] 
where $*$ indicates that the entry $1+\alpha_j-\alpha_j$ is omitted \cite[\S 16.8(ii)]{NIST}. 
One has the differentiation formulas \cite[\S 16.3(i)]{NIST}
\[   \frac{d}{dz} {}_0F_r(-;\alpha_1+1,\ldots,\alpha_r+1;z) = \frac{1}{(\alpha_1+1)\cdots(\alpha_r+1)} {}_0F_r(-;\alpha_1+2,\ldots,\alpha_r+2;z), \]
and for $1 \leq  j \leq r$
\[   \frac{d}{dz} z^{\alpha_j} {}_0F_r(-;\alpha_1+1,\ldots,\alpha_r+1;z) = \alpha_j z^{\alpha_j-1} {}_0F_r(-;\alpha_1+1,\ldots,\alpha_j,\ldots,\alpha_r+1;z). \]

\section{Applications}
The functions ${}_0F_r(-; \alpha_1+1,\ldots,\alpha_r+1;-z)$ appear as the limit in Theorems \ref{thm:JacPin}, \ref{thm:Ln1}, \ref{thm:Knu} (for $r=2$)
and \ref{thm:G}. It also appears in Theorem \ref{thm:JacAng} for $r=2$ with a quadratic variable and in Theorem \ref{thm:SoTu} with the variable $x^r$. It is known that the zeros of the Jacobi-Pi\~neiro polynomials are all in $[0,1]$ and the zeros of the multiple Laguerre polynomials of the first kind are all in $[0,\infty)$, since
the $r$ weights form an AT-system \cite[Thm. 23.1.4]{Ismail}. As a result of Hurwitz' theorem, we can therefore conclude that all the zeros
of ${}_0F_r(-; \alpha_1+1,\ldots,\alpha_r+1;-z)$ are positive. Let us denote these zeros by $(f_k)_{k=1,2,\ldots}$, then as a consequence of
Theorem \ref{thm:JacPin} we have that the $k$th zero $x_{k,\vec{n}}$ of the Jacobi-Pi\~neiro polynomial $P_{\vec{n}}$ has the asymptotic behavior
\[   \lim_{n \to \infty} n^{r+1} x_{k,\vec{n}} = \frac{f_k}{q_1\cdots q_r}, \qquad k=1,2,\ldots, \]
when $\vec{n} = (\lfloor q_1n \rfloor, \ldots, \lfloor q_r n\rfloor )$.
In a similar way one finds for the zeros of the multiple Laguerre polynomials of the first kind
\[   \lim_{n \to \infty} n^{r} x_{k,\vec{n}} = \frac{f_k}{q_1\cdots q_r}, \qquad k=1,2,\ldots,  \]
but for the zeros of the multiple Laguerre polynomials of the second kind one has
\[   \lim_{n \to \infty} n x_{k,\vec{n}} = \frac12 \left( \frac{j_k}{q_1c_1+\cdots+q_rc_r} \right)^2, \qquad k=1,2,\ldots, \]
where $(j_k)_{k=1,2,\ldots}$ are the positive zeros of the Bessel function $J_\alpha$.
Similar results hold for the zeros $x_{k,n}$ $(1 \leq k \leq n)$ of the multiple orthogonal polynomials $p_n$ for the modified Bessel functions $K_\nu$ and $K_{\nu+1}$:
\[  \lim_{n \to \infty} n x_{k,n} = f_k, \]
where $(f_k)_{k=1,2,\ldots}$ are the zeros of ${}_0F_2(-;\alpha+1,\alpha+\nu+1;-z)$
and in general for the multiple orthogonal polynomials $P_n$ with the Meijer $G$-functions
\[   \lim_{n \to  \infty} n x_{k,n} = f_k,  \]
where $(f_k)_{k=1,2,\ldots}$ are now the zeros of ${}_0F_r(-;\nu_1+1,\ldots,\nu_r+1;-z)$.

The Mehler-Heine formulas give the asymptotic behavior, after appropriate scaling, of the multiple orthogonal polynomials near the origin. This detailed
asymptotic behavior is crucial when one wants to obtain the global (or uniform) asymptotic behavior of the multiple orthogonal polynomials
using the Riemann-Hilbert problem \cite{WVAGerKuijl} and the steepest descent method of Deift and Zhou for oscillatory Riemann-Hilbert problems
\cite{DeiftZhou}. In order to get uniform estimates, one needs to introduce local parametrices around the endpoints of the support of the asymptotic
zero distribution. Around the hard edges, i.e., those endpoints of the support of the zero distribution which are also endpoints of the supports of the weights $w_1,\ldots,w_r$, one will be needing a Riemann-Hilbert problem for which the solution contains the functions
${}_0F_r(-;\alpha_1+1,\ldots,\alpha_r+1;-z)$
and all the other solutions of the differential equation \eqref{diffeq}. This was already used for orthogonal polynomials on $[-1,1]$ with a weight function that behaves near $\pm 1$ as the Jacobi weights, and the parametrices around $\pm 1$ use a Riemann-Hilbert problem involving Bessel functions
\cite[pp.~365--368]{KMcLVAVL} \cite[\S 14]{Kuijl}. The same local parametrix was used for Laguerre-type orthogonal polynomials \cite[\S 3.7]{Vanlessen}.
A local parametrix around the origin was used for Jacobi-Angelesco polynomials in \cite{DeschKuijl1} and \cite[\S II.B]{DeschKuijl2}, but the relation with the generalized Bessel function was not explicit since they used the differential equation
\[ zq''' - \beta q'' - \tau q' + q =0,  \]  
which for $\tau = 0$ has three independent solutions of the form ${}_0F_2(-;1/2,-\beta/2;-z^2/8)$, $z {}_0F_2(-;3/2,(-\beta+1)/2;-z^2/8)$ and
$z^{\beta+2} {}_0F_2(-;(\beta+3)/2,(\beta+4)/2;-z^2/8)$.

\section{Concluding remarks}

The classical Mehler-Heine asymptotics for the Jacobi polynomials and the Laguerre polynomials near the endpoints of the intervals is in terms
of the Bessel function $J_\nu$, where the order $\nu$ is related to the behavior of the weight function near the endpoint $c$, i.e., 
$w(x) = \mathcal{O}((c-x)^\nu)$ as $x \to c$.
Our Mehler-Heine asymptotics describes the behavior of multiple orthogonal polynomials near an endpoint of the intervals of orthogonality.
For multiple orthogonal polynomials there are more weights, and our examples show that there is an important difference when weights
have a common endpoint at $0$, because then the asymptotic behavior is in terms of a `generalized' Bessel function which depends on the
behavior of these weights near that endpoint. The fact that $r$ weights have a common endpoint is crucial to get the `generalized' Bessel function.
For Jacobi-Angelesco polynomials one has also endpoints at $1$ and $-1$, but these are endpoints of only one weight function. The Mehler-Heine
asymptotics around those points would be in terms of Bessel functions $J_\alpha$ (for the endpoint $-1$) and $J_\beta$ (for the endpoint $1$).
For Jacobi-Pi\~neiro polynomials one also has a common endpoint at $1$, but the weights all have the same behavior $\mathcal{O}((1-x)^\beta)$ near that
endpoint. The asymptotic behavior around the endpoint $1$ has not been analyzed yet, but I expect it will be in terms of the Bessel function
$J_\beta$ rather than in terms of the generalized Bessel function.

The Mehler-Heine asymptotics for Jacobi-Angelesco polynomials on $[-a,0] \cup [0,1]$ with $a > 0$ but $a \neq 1$ was obtained by Takata in \cite{Takata1}.
When the two intervals are not of equal length, then the Mehler-Heine asymptotics is in terms of the Bessel function. This is because for $a \neq 1$
the zeros of the Jacobi-Angelesco polynomials accumulate on $[-a,0] \cup [a^*,1]$ with $a^*>0$ when $a<1$, or on $[-a,a^*] \cup [0,1]$ with $a^*<0$ 
when $a>1$. These intervals are not touching and $0$ is a hard edge for only one interval, whereas $a^*$ is a soft edge for the other interval.
For one hard edge one gets Bessel functions, for two hard edges one gets the limit function from Theorem~\ref{thm:JacAng} involving the
generalized Bessel function.

For the Jacobi-Angelesco polynomials (Theorem \ref{thm:JacAng}) we only considered the asymptotic behavior for the polynomials $P_{n,n}$, and in Theorems \ref{thm:Knu}--\ref{thm:Inu}
we only considered the asymptotic behavior for the multiple orthogonal polynomials $P_{n,n}$, $P_{n+1,n}$. This is because we do not have an
explicit expression of the multiple orthogonal polynomials $P_{n,m}$ when $|n-m| >1$. For the Jacobi-Angelesco polynomials, however,  we have
a Rodrigues formula of the form
\begin{multline*}    (1+x)^\alpha x^\beta (1-x)^\gamma P_{n,n+k}(x) \\
    = C_n \frac{d^n}{dx^n} \Bigl( (1+x)^{n+\alpha} x^{n+\beta} (1-x)^{n+\gamma} P_{0,k}^{(n+\alpha,n+\beta,n+\gamma)}(x) \Bigr),
\end{multline*}
where $P_{0,k}^{(n+\alpha,n+\beta,n+\gamma)}$ is the orthogonal polynomial of degree $k$ for the weight function
$(1+x)^{n+\alpha} x^{n+\beta} (1-x)^{n+\gamma}$ on $[0,1]$, so one needs to take the Mehler-Heine asymptotics for this polynomial into account as well.
This can be done for fixed $k$ but the analysis becomes more difficult when $k$ is allowed to go to infinity as $n \to \infty$. In that case
it is possible that the zeros accumulate in $[-1,a^*] \cup [0,1]$ with $a^*<0$, so that one again deals with a soft edge at $a^*$ and a hard edge at $0$,
so that it is expected to find Mehler-Heine asymptotics around $0$ in terms of the Bessel function $J_\beta$.


\end{document}